\documentclass[a4paper,11pt,reqno]{article}
\usepackage{latexsym,amsmath,amssymb,theorem}
\usepackage[all]{xy}

\usepackage{color} 
\usepackage{hyperref}
\theoremstyle{change}

{\theorembodyfont{\itshape}
\newtheorem{lemma}{lemma}
\newtheorem{proposition}[lemma]{Proposition}
\newtheorem{corollary}[lemma]{Corollary}
\newtheorem{theorem}[lemma]{Theorem}}
{\theorembodyfont{\upshape}

\newtheorem{remark}[lemma]{Remark}
\newtheorem{remarks}[lemma]{Remarks}
\newenvironment{proof}{\noindent \bf Proof:  \rm}{\hspace*{\fill}
$\square$ \vskip8pt}

   \newenvironment{Facts}{\vskip8pt\noindent \bf Facts \rm}{\vskip8pt}
\newcommand{\T}{\mathbb{T}}

\newcommand{\mcd}{\mathcal{D}}

\newcommand{\msa}{\mathsf{A}}

\newcommand{\msc}{\mathsf{C}}
\newcommand{\msd}{\mathsf{D}}

\newcommand{\BC}{\mathcal{C}}
\newcommand{\BD}{\mathcal{D}}

\newcommand{\xra}{\xrightarrow}

\newcommand{\Mon}{\mathbf{Mon}}

\newcommand{\Ring}{\mathbf{Ring}}

\newcommand{\im}{\mathrm{im}}

\newcommand{\ot}{\otimes}

\newcommand{\Ab}{\mathbf{Ab}}

\begin{document}


  \title{Colimits  of monoids }

 \author{ 
 Hans-E. Porst\thanks{Permanent address: 
 Besselstr. 65, 28203 Bremen, Germany.} \\
\small Department of Mathematical Sciences, University of Stellenbosch, \\ \small Stellenbosch, South Africa. \\
 {\tt\small porst@uni-bremen.de} 
}

  
  \date{}
\maketitle

\begin{abstract}
If $\BC$ is a cocomplete monoidal category in which tensoring from both sides preserves coequalizers, then the category $\Mon\BC$ of monoids over $\BC$ is cocomplete. The same holds if $\BC$ has regular factorizations and tensoring only preserves regular epimorphisms. As an application a lifting theorem for an adjunction with a monoidal right adjoint  to an adjunction  between the respective  categories of monoids is proved.

\vskip 8pt
\noindent {\bf MSC 2000}: {Primary 18D10, Secondary 18A30}
\vskip 8pt
\noindent {\bf Keywords:} {Monoids in monoidal categories, coequalizers, (regularly) monadic functors}

\end{abstract}
\section{Introduction}
While limits in categories of monoids in $\Mon\BC$, the category of monoids over a monoidal category $\BC$, are well know, since the forgetful functor $\Mon\BC\xra{ |-|}\BC$ creates limits (see e.g. \cite{HEP_QM}), 
 not much seems to be known about the existence of colimits in $\Mon\BC$. Assuming that $\BC$ is cocomplete and $\ot$ preserves colimits, it is shown in \cite{SS} that  certain pushouts exist in $\Mon\BC$, while \cite{HEP_QM} deals with the more restrictive situation
 that $\BC$ is locally finitely presentable and $\ot$ preserves directed colimits: Then $\Mon\BC$ is locally presentable and, hence, cocomplete.

In this note we will consider the case where $\BC$ is cocomplete and $\Mon\BC\xra{ |-|}\BC$ is monadic. Then, by a well known result on monadic functors (see e.g. \cite[4.3.4]{Bor}), the category $\Mon\BC$ will be cocomplete provided that it has coequalizers. 

An application of this criterion becomes possible by a result on the existence of coequalizers in $\Mon\BC$ in  the preprint \cite{AEM}. Unfortunately, neither the publication \cite{AEM2} with a similar title nor the actual arXiv-version contains
this result such that it is nearly  impossible to become aware of it.
Since the author of this note is convinced that this result on the existence of coequalizers in $\Mon\BC$ should become part of the ``official'' literature and the authors of that preprint have no intention to publish this result \cite{LK}, its proof is included in Section \ref{sec:coeq} below. We add a short remark concerning what might have been the motivation for the construction described in that paper.
A similar comment is in place with respect to  the interesting application concerning left adjoints of functors between categories of monoids induced by a monoidal functor between the base categories,  presented  as Theorem  \ref{thm:T} below. Though  this theorem generalizes considerably  a similar one in \cite{AEM}, our proof uses  crucial elements of that proof.  We make clear, however, that what appears as an ad hoc construction  in \cite{AEM}
 in fact is  based on general principles for the lifting of  adjunctions.

Finally a word concerning terminology: by the phrase \em $\ot$ preserves colimits (of some type) \em  in the monoidal category $\BC$  we mean that, for each $C$ in $\BC$, the functors $C\ot -$ and $-\ot C$ preserve these colimits.


\section{Coequalizers in $\Mon\BC$}\label{sec:coeq}
Somewhat surprisingly, coequalizers in $\Mon\BC$ can be constructed the same way as in the special case of the category of unital rings, that is, as in $\Ring =\Mon\Ab$ where $\Ab$ denotes the monoidal category of abelian groups. We therefore recall this simple construction as follows: If 
\begin{equation*} %
\begin{aligned}
\xymatrix@=2.5em{
R\ar@<.5ex>[r]^{\alpha}\ar@<-.5ex>[r]_{\beta}\  & S\ar[r]^\pi&Q 
}
\end{aligned}
\end{equation*}
is a coequalizer diagram in $\Ring$, then $ = S/I_{\alpha,\beta}$, where $I_{\alpha,\beta}$ is the two-sided ideal generated by the image $\im(\alpha-\beta)$ in $\Ab$.
In other words, $I_{\alpha,\beta} = S\cdot \im(\alpha-\beta)\cdot S$.  

Denoting by  $m_S$ the multiplication of $S$ and, for any $\Ab$-morphism $X\xra{ \gamma}S$, by $\Lambda_\gamma$  the $\Ab$-morphism
\begin{equation*}
S\otimes X\otimes S\xra{ S\ot \gamma\ot S}S\ot S\ot S\xra{m_S\ot S}S\ot S\xra{ m_S}S
\end{equation*}
one obtains $I_{\alpha,\beta}=\im{\Lambda_{\alpha -\beta}} = \im(\Lambda_{\alpha}-\Lambda_{\beta})$ such that 
%
%
\begin{equation*} %
\begin{aligned}
\xymatrix@=2.5em{
S\ot R\ot S\ar@<.5ex>[r]^{\ \ \ \ \ \ \ \Lambda_\alpha}\ar@<-.5ex>[r]_{\ \ \ \ \ \ \ \Lambda_\beta}\  & S\ar[r]^\pi&Q 
}
\end{aligned}
\end{equation*}
is a coequalizer diagram in $\Ab$.


Following \cite{AEM} we will show that this essentially can  be generalized to an arbitrary monoidal category $\BC$. 
 Starting with a $\BC$-monoid 
 $(A,m_A,e_A)$  and   a $\BC$-morphism $X\xra{ \alpha} A$ we denote, 
using the same notation as above, by $\Lambda_\alpha$  the $\BC$-morphism
\begin{equation*}
A\otimes X\otimes A\xra{ A\ot \alpha\ot A}A\ot A\ot A\xra{m_A\ot A}A\ot A\xra{ m_A}A
\end{equation*}

\begin{Facts}\label{fact:1}\rm
The following facts are easy to verify:
\begin{enumerate}
\item For every monoid morphism $A\xra{ \tau} C$ the following hold (up to identification of $X$ and $I\ot X\ot I$):
\begin{eqnarray*}
\alpha = \Lambda_\alpha\circ (e_A\ot X\ot e_A) \text{ \ \ and \ \ } \tau\circ\Lambda_\alpha = \Lambda_{\tau\circ\alpha}\circ (\tau\ot X\ot \tau) 
\end{eqnarray*}

The following is then are immediate consequences. 
\begin{enumerate}
\item If $X\xra{ \beta} A$ is another $\BC$-morphism then 
\begin{equation} \label{eq:5}
\tau\circ\Lambda_\alpha  =\tau\circ\Lambda_\beta \iff \tau\circ\alpha  =\tau\circ\beta
\end{equation}
where the implication $\Rightarrow$ even holds, if $\tau$ only is a $\BC$-morphism. 
\item For every monoid morphism $A\xra{ \gamma}B$ one has 
\begin{equation*}\label{eqn::c}
\gamma\circ\Lambda_{\alpha} = \Lambda_{\gamma\circ\alpha}\circ (\gamma\ot X\ot \gamma)
\end{equation*}
\end{enumerate}
\end{enumerate}
\end{Facts}

\begin{lemma}\label{prop1}
Let $\BC$ be a monoidal category with coequalizers preserved by $\ot$ and let $A$ be a monoid in $\BC$.
If   $\alpha, \beta\colon X\xra{}A$ are   $\BC$-morphisms and
\begin{equation*} 
\begin{aligned}
\xymatrix@=2.5em{
A\ot X\ot A\ar@<.5ex>[r]^{\ \ \ \ \ \ \ \Lambda_\alpha}\ar@<-.5ex>[r]_{\ \ \ \ \ \ \ \Lambda_\beta}\  & A\ar[r]^\pi&Q 
}
\end{aligned}
\end{equation*} is a coequalizer diagram in $\BC$, then 
 $Q$ carries a  unique monoid structure such that $\pi$ is a monoid morphism.
 \end{lemma}

\begin{proof}
Observe first that the following diagrams commute.
\begin{equation}\label{diag:basic}
\begin{aligned}
\xymatrix@=2.5em{
 A\ot A\ot X\ot A    \ar@<.5ex>[r]^{\ \ \ \ \ \ \ A\ot\Lambda_\alpha}\ar@<-.5ex>[r]_{\ \ \ \ \ \ \ A\ot\Lambda_\beta}
 \ar[d]_{ m_A\ot X\ot A }
 &{A\ot A } \ar[d]^{m_A }  \\
A\ot X\ot A \ar@<.5ex>[r]^{\ \ \ \ \ \ \ \Lambda_\alpha}\ar@<-.5ex>[r]_{\ \ \ \ \ \ \ \Lambda_\beta}         &   A
}\hspace{1cm}
\xymatrix@=2.5em{
  A\ot X\ot A  \ot A \ar@<.5ex>[r]^{\ \ \ \ \ \ \ \Lambda_\alpha\ot A}\ar@<-.5ex>[r]_{\ \ \ \ \ \ \ \Lambda_\beta\ot A}\ot A \ar[d]_{A\ot X\ot  m_A }
  &{A\ot A } \ar[d]^{m_A }  \\
A\ot X\ot A \ar@<.5ex>[r]^{\ \ \ \ \ \ \ \Lambda_\alpha}
\ar@<-.5ex>[r]_{\ \ \ \ \ \ \ \Lambda_\beta}  &   A
} 
\end{aligned}
\end{equation}
Since by assumption $A\ot A\xra{ A\ot \pi} A\ot Q$ is a coequalizer of  $A\ot \Lambda_\alpha$ and $ A\ot \Lambda_\beta$, one obtains a unique $\BC$-morphism $A\ot Q\xra{m }Q$ satisfying
\begin{equation}\label{3}
m\circ (A\ot \pi) = \pi\circ m_A
\end{equation}
By commutativity of Diagrams \eqref{diag:basic} and Equation \eqref{3}  the following  diagram  commutates.
\begin{equation*}
\begin{aligned}
\xymatrix@=3em{
 A\ot X\ot A\ot A    \ar@<.5ex>[r]^{\ \ \ \ \ \ \ \Lambda_\alpha\ot A}\ar@<-.5ex>[r]_{\ \ \ \ \ \ \ \Lambda_\beta\ot A}
 \ar[d]_{A\ot X\ot A\ot\pi  }&{ A\ot A}\ar[r]^{m_A} \ar[d]^{A\ot \pi } & A\ar[d]^{\pi } \\
  A\ot X\ot A\ot Q   \ar@<.5ex>[r]^{\ \ \ \ \ \ \ \Lambda_\alpha\ot Q}\ar@<-.5ex>[r]_{\ \ \ \ \ \ \ \Lambda_\beta\ot Q}        &   A\ot Q\ar[r]_m&Q
}
\end{aligned}
\end{equation*}
Since $A\ot X\ot A\ot\pi $ is a (regular) epimorphism and $A\ot Q\xra{ \pi\ot Q} Q\ot Q$ is the coequalizer of $\Lambda_\alpha\ot Q$ and   $\Lambda_\beta\ot Q$  by 
assumption, 
there exists a unique $\BC$-morphism $Q\ot Q\xra{ m_Q}Q$ satisfying $m_Q\circ (\pi\ot Q) =m$.

Using the facts that $A$ is a monoid and $\pi$ is a regular epimorphism (hence $\ot^n\pi$ an epimorphism for each $n$) one shows easily that $(Q,m_Q, \pi\circ e_A)$ is a monoid and and $\pi$ a monoid morphism.
\end{proof}

\begin{corollary}\label{corr:}
If   $\alpha, \beta\colon X\xra{}A$  and $\alpha', \beta'\colon X'\xra{ }A$ are $\BC$-morphisms such that, for every monoid morphism $A\xra{ \tau}C$, 
\begin{equation}\label{eq:2}
\tau\circ\alpha  =\tau\circ\beta \iff \tau\circ\alpha'  =\tau\circ\beta'.
\end{equation}
Then the coequalizers of $(\Lambda_\alpha, \Lambda_\beta)$ and $(\Lambda_{\alpha'}, \Lambda_{\beta'})$ coincide.
\end{corollary}
\begin{proof}
Let $A\xra{ \pi}E$ and $A\xra{ \pi'}E'$  be  coequalizers of  $(\Lambda_\alpha, \Lambda_\beta)$ and  $(\Lambda_{\alpha'}, \Lambda_{\beta'})$, respectively. Since $\pi$ and $\pi'$ are monoid morphisms by Lemma \ref{prop1}, one obtains by the equivalences \eqref{eq:5} and \eqref{eq:2} the  equivalences   
$$\pi\circ\Lambda_\alpha = \pi\circ\Lambda_\beta \iff \pi\circ \alpha = \pi\circ\beta \iff \pi\circ \alpha' = \pi\circ\beta' \iff \pi\circ\Lambda_{\alpha'} = \pi\circ\Lambda_{\beta'},$$
which implies the claim.
\end{proof}

\begin{remark}\label{rem:bar}\rm
Equivalence \eqref{eq:2} holds in particular, if $\alpha'$ and $\beta'$ are the homomorphic extensions $TX\rightarrow A$ of $\alpha$ and $\beta$.
\end{remark}

\begin{theorem}[\cite{AEM}]\label{thm:1}
 Let $\BC $ be a monoidal category with coequalizers preserved by $\otimes$. Then the category $\Mon\BC$ has coequalizers and the forgetful functor $\Mon\BC\xra{ }\BC$ preserves regular epimorphisms.
  \end{theorem}
In more detail: 
If   $\alpha, \beta\colon X\xra{}A$ are   monoid morphisms and
\begin{equation*} 
\begin{aligned}
\xymatrix@=2.5em{
A\ot X\ot A\ar@<.5ex>[r]^{\ \ \ \ \ \ \ \Lambda_\alpha}\ar@<-.5ex>[r]_{\ \ \ \ \ \ \ \Lambda_\beta}\  & A\ar[r]^\pi&Q 
}
\end{aligned}
\end{equation*} is a coequalizer diagram in $\BC$ then, with  the monoid structure   of Lemma \ref{prop1} on $Q$,   then 
\begin{equation*} 
\begin{aligned}
\xymatrix@=2.5em{
X\ar@<.5ex>[r]^{\alpha}\ar@<-.5ex>[r]_{\beta}\  & A\ar[r]^\pi&Q 
}
\end{aligned}
\end{equation*}
 is a coequalizer diagram in $\Mon\BC$.

\begin{proof}
Let $\alpha,\beta\colon D\xra{ }A$ be monoid morphisms and $A\xra{ \pi} Q$  the coequalizer of $\Lambda_\alpha$ and $\Lambda_\beta$ in $\BC$. By Lemma \ref{prop1}  $Q$ carries a unique monoid structure such that $\pi$ is a monoid morphism. If $A\xra{ \tau}C$ is a monoid morphism such that $\tau\circ\alpha =\tau\circ\beta$, then $\tau\circ\Lambda_\alpha  =\tau\circ\Lambda_\beta$ by the equivalence \eqref{eq:5}, such that there exists a unique  $\BC$-morphism $Q\xra{\sigma}C$ with $\sigma\circ\pi = \tau$. It remains to prove that $\sigma$ is a monoid morphism. But this is clear since $\pi$ is a regular epimorphism and, hence, $\pi\ot\pi$ is an epimorphism.
\end{proof}

\begin{remark}\label{rem:5}\rm
Given two parallel pairs of morphisms in a category with coequalizers
\begin{equation*} 
\begin{aligned}
\xymatrix@=2.5em{
X_1\ar@<.5ex>[r]^{f_1}\ar@<-.5ex>[r]_{g_1}& A 
}\hspace{1cm}
\xymatrix@=2.5em{
X_2\ar@<.5ex>[r]^{f_2}\ar@<-.5ex>[r]_{g_2} & A
}
\end{aligned}
\end{equation*}
one obtains a multiple coequalizer of the morphisms $f_1,g_1, f_2, g_2$ as follows: Form  the coequalizer $A\xra{ q_1}Q_1$ of $f_1,g_1$ and then the coequalizer  $Q_1\xra{ q_2} Q$ of $q_1\circ f_2, q_1\circ g_2$; then $A\xra{ q_1}Q_1\xra{q_2 }Q$ is the required multiple coequalizer. 

In particular, every category with coequalizers has such multiple coequalizers and any such is a composite of ordinary ones. 
\end{remark}


\section{Applications} 

\subsection*{Monadicity}
Applying the result above we first provide two similar monadicity criteria for the forgetful functor $|-|\colon \Mon\BC\xra{ }\BC$, provided that this has a left adjoint.

\begin{proposition}\label{thm:mon1}
Let $\BC $ be a monoidal category with regular factorizations and denumerable coproducts and assume that these as well as  regular epimorphisms are preserved by  $\otimes$. Then the forgetful functor $|-|\colon \Mon\BC\xra{ }\BC$ is regularly monadic.
\end{proposition}

\begin{proof}
By the standard construction of free monoids $|-|$ has a left adjoint. 
Monadicity follows by the Beck-Par\'e-Theorem 
(the argument given in \cite{Pare} for the case of semigroups applies by replacing $\times$ by $\ot$; it only requires that  $\ot$ preserves regular epimorphisms\footnote{In fact somewhat less is needed here, namely that for each $C$ in $\BC$ the functors $C\ot -$ and $-\ot C$ map regular epimorphisms to epimorphisms.}).

The respective monad $\T$ acts on a morphism $f$ by $\T f = \coprod_n \ot^n f$, hence, maps regular epimorphisms to regular epimorphisms by assumption.
\end{proof}

\begin{proposition}\label{thm:mon2}
Let $\BC $ be a  monoidal category with coequalizers and assume that these are preserved by  $\otimes$ and that $|-|\colon \Mon\BC\xra{ }\BC$ has a left adjoint.
Then the functor $|-|\colon \Mon\BC\xra{ }\BC$ is monadic (and preserves regular epimorphisms).
\end{proposition}

\begin{proof}
Monadicity  follows as above.
 $|-|$ preserves regular epimorphisms    by  Theorem \ref{thm:1}. 
\end{proof}

\begin{remarks}\label{rem:}\rm
Note the differences between these  results:
\begin{enumerate}
\item  Proposition \ref{thm:mon1} requires the free monoid functor to be given in the standard way, while Proposition \ref{thm:mon2} works for an arbitrary  left adjoint of $|-|$ (see e.g. \cite{Lack} for important examples). 
\item Proposition \ref{thm:mon1} requires $\BC$ to have regular factorizations, while  Proposition \ref{thm:mon2} only requires the existence of coequalizers in $\BC$.
\item Proposition \ref{thm:mon2} requires $\ot$ to preserve  coequalizers, while  Proposition \ref{thm:mon1} only 
requires preservation of regular epimorphisms.
\end{enumerate}

Moreover, the assumption on coequalizers in Proposition \ref{thm:mon2} can be restricted to reflexive coequalizers if $\BC$ has binary coproducts which are preserved by $\ot$. In fact, such a category 
has coequalizers and a functor  $F$ on $\BC$ preserves those if and only if $\BC$ has reflexive coequalizers preserved by $F$: For any pair of morphisms $f,g\colon C\xra{ }D$ a morphism $D\xra{q}Q$ is a coequalizer of $f$ and $g$ if and only if $D\xra{q}Q$ is a coequalizer of the  reflexive pair $\bar{f},\bar{g}\colon C+D\xra{ }D$ where $C+D\xra{ \bar{f}}D$ is the morphism with components $f$ and $id_D$.
\end{remarks}

\subsection*{General Colimits in $\Mon\BC$}

Not much seems to be known about the existence of colimits in $\Mon\BC$. The only results we are aware of are
\begin{itemize}
\item If $\BC$ is cocomplete and $\ot$ preserves colimits, then the category 
$\Mon\BC$ has all pushouts of the form
\begin{equation*}
\begin{aligned}
\xymatrix@=2em{
   FC  \ar[r]^{Ff  }\ar[d]_{  }&{ FD} \ar[d]^{ }  \\
  X  \ar[r]_{ }          &    P 
}
\end{aligned}
\end{equation*}
where $C\xra{ f}D$ is a
 morphism in $\BC$
 and 
$F\colon\BC\xra{ }\Mon\BC$ is the free monoid functor (see \cite{SS}).
\item If $\BC$ is locally $\lambda$-presentable and $\ot$ preserves $\lambda$-directed colimits, then $\Mon\BC$ is locally presentable and, hence, cocomplete in particular (see \cite{HEP_QM}).
\end{itemize}

Using the monadicity criteria above we obtain the following results and so generalize  the result of \cite{SS} substantially.

\begin{proposition}\label{thm:col1} In the situation of Proposition \ref{thm:mon1}
$\Mon\BC$ has coequalizers and all other colimits which exist in $\BC$. In particular, if $\BC$ is cocomplete then so is $\Mon\BC$.
\end{proposition}
\begin{proof}
Since every category with regular factorizations has coequalizers (see   \cite[20.33]{AHS}) the result follows immediately.
\end{proof}

\begin{proposition}\label{thm:col2}
In the situation of Proposition \ref{thm:mon2} the category $\Mon\BC$ has coequalizers. Moreover,  if $\BC$ is cocomplete then so is $\Mon\BC$.
\end{proposition}
\begin{proof}
The first result follows from Theorem \ref{thm:1} and implies the second by \cite[4.3.4]{Bor}. 
\end{proof}

\subsection*{Lifting adjunctions}\label{sec:Tamb}

Let $R\colon \BC\rightarrow\BD$ be a monoidal functor. It is well known that $R$ induces a functor $\bar{R}\colon\Mon\BC\rightarrow\Mon\BD$ such that the diagram
\begin{equation*}
\begin{aligned}
\xymatrix@=2em{
  \Mon\BC   \ar[r]^{ \bar{R} }\ar[d]_{  }&\Mon\BD \ar[d]^{ }  \\
 \BC   \ar[r]_{ R}          &    \BD 
}
\end{aligned}
\end{equation*}
commutes, where the vertical arrow denote the respective forgetful functors (denoted by $|-|$ if necessary).
It is of quite some interest (see e.g. \cite{Ver}, \cite{PS}) to know under which conditions the functor  $\bar{R}$ has a left adjoint if $R$ has one.

A standard approach to this problem would be to apply Dubuc's Adjoint Triangle Theorem, which would require both forgetful functors to have left adjoints and $\Mon\BC$ to have coequalizers of reflexive pairs. Tambara \cite{Tam} claimed without a proof that it suffices to assume that $\BC$ is cocomplete and that $\ot$ preserves all colimits. A proof of this claim is contained in \cite{AEM}.
The following is a generalization of this result in that we do not assume the free monoids over $\BC$ to be given by MacLane's standard construction.
 
\begin{theorem}\label{thm:T}
Let $R\colon \BC\rightarrow\BD$ be a monoidal functor with left adjoint $L$. Assume that $\BC$ has coequalizers which are preserved by $\ot$ and that the forgetful functor $\Mon\BC\xra{ |-|}\BC$  has left adjoint $T$.
Then the functor $\bar{R}\colon\Mon\BC\rightarrow\Mon\BD$  has a left adjoint.
\end{theorem}
\begin{proof}
We use the following notations.
\begin{enumerate}
\item Unit and counit of the adjunction $T\dashv |-|$ are denoted by $ \xi$ and $\zeta$, respectively. For every $C$ in $BC$  $m_{TC}$ and $e_{TC}$ are the multiplication and unit, respectively, of the free monoid $TC$. 
\item $\Phi$ denotes the multiplication and $\phi$ the unit of the monoidal structure of $R$; $\Psi$ and $\psi$ denote the opmonoidal structure of $L$. $\kappa\colon id\Rightarrow RL$ and $\lambda\colon LR\Rightarrow id$
denote the unit and counit, respectively, of the adjunction $L\dashv R$.
\item Unit and counit of the adjunction $F_1:=TL\dashv R|-|=:P_1$ are denoted by $\eta $ and $\epsilon$, respectively. In particular, for any $\BD$-object~$D$,
\begin{equation}
 \eta_D = D\xra{\kappa_D}RLD\xra{ R\xi_{LD}}R|TLD|
\end{equation}
\end{enumerate}

Following the analysis of adjoint triangles in \cite{Thol} one should try to obtain the left adjoint of $\bar{R}$ as follows:
\begin{enumerate}
\item Find in $\Mon\BC$, for each $\BD$-monoid $\msd:=(D,m_D,e_D)$, a suitable morphism $$F_1 D = TLD\xra{ \pi_\msd} L_\msd.$$
\item Find  in $\Mon\BD$ a morphism $\msd\xra{\gamma_\msd }\bar{R}L_\msd  $  with $$|\gamma_\msd| = D\xra{ \eta_D}P_1F_1D\xra{R|\pi_\msd|}\bar{R}L_\msd.$$
\end{enumerate}
 $\gamma_\msd$ so defined  has the potential of being $\bar{R}$-universal for $\msd$, hence the family $\gamma =(\gamma_\msd)_\msd$ to be the unit of the desired adjunction.

 Now $\pi_\msd$ has to be an epimorphism since the forgetful functor $\Mon\BD\rightarrow\BD$ is faithful (see \cite{Thol}). A natural choice of $\pi_\msd$, thus, would be to consider a (multiple) coequalizer of $\Mon\BC$-morphisms which 
in some way reflect the monoid structure of $\msd = (D,m_\msd,e_\msd)$ and the monoidal structure of $R$ (equivalently, the opmonoidal structure of $L$). 
We therefore use the following morphisms (see also \cite{AEM}).
\begin{enumerate}
\item \label{c1}$\alpha_1^\msd = LI_\BD\xra{ \psi}I_\BC\xra{e_{TLD}}|TLD|$ \\ 
 $\beta_1^\msd = LI_\BD\xra{ Le_\msd}LD\xra{\xi_{LD}}|TLD|$
\item  \label{c2} $\alpha_2^\msd = L(D\ot D)\xra{ \Psi_{D,D}}LD\ot LD\xra{\xi_{LD\ot LD} }|TLD|\ot |TLD|\xra{ m_{TLD}}|TLD|$  \\  
$\beta_2^\msd = L(D\ot D)\xra{Lm_\msd}LD\xra{ \xi_{LD}}|TLD|$
\end{enumerate}
and consider the homomorphic extensions of these maps, that is, the monoid morphisms
\begin{enumerate}
\setcounter{enumi}{2}
\item \label{c5}$\bar{\alpha}_1^\msd, \bar{\beta}_1^\msd\colon TLI_\BD\xra{ } TLD$ with $\bar{\alpha}_1^\msd\circ \xi_{LI_\BD}=\alpha_1^\msd$ and $\bar{\beta}_1^\msd\circ \xi_{LI_D}=\beta_1^D$
\item \label{c6}$\bar{\alpha}_2^\msd, \bar{\beta}_2^\msd\colon TL(D\ot D)\xra{ } TLD$ with $\bar{\alpha}_2^\msd\circ \xi_{L(D\ot D)}=\alpha_2^\msd$ and $\bar{\beta}_2^\msd\circ \xi_{L(D\ot D)}=\beta_2^D$.
\end{enumerate}
Now let $ TLD\xra{ \pi_\msd}L_\msd$ be the multiple coequalizer (which exists by assumption --- see Remark \ref{rem:5}) 
of the morphisms $\bar{\alpha}_1^\msd$, $\bar{\beta}_1^\msd$, $\bar{\alpha}_2^\msd$, $\bar{\beta}_2^\msd$
 in $\Mon\BC$.
 
 According to step 2.  we check that the  following $\BD$-morphism is a morphism of $\BD$-monoids $\msd\xra{ }\bar{R}L_\msd$. 
\begin{equation*}
\msd\xra{ \gamma_\msd}\bar{R}L_\msd:=\msd\xra{\eta_D}\bar{R}TLD\xra{ R\pi_\msd}\bar{R}L_\msd
\end{equation*}
Compatibility with the multiplications is equivalent to commutativity of  the outer frame of the following diagram,   which is clear  by standard arguments for monoidal functors and the fact that $\pi_\mcd$ coequalizes $\bar{\alpha}_2^\msd$ and $\bar{\beta}_2^\msd$.

Preservation of  units follows by a similar argument.

\begin{equation*}
\begin{aligned}\small
\xymatrix@=3.5em{
   D\ot D\ar@/^2pc/@{->}[rr]_{\eta_D\ot \eta_D}\ar[rdd]_{\kappa_{D\ot D}} \ar[r]_{\!\!\!\!\!\!\!\!\!\!\kappa_D\ot \kappa_D} \ar[ddd]_{ m_D } & {RLD\ot RLD }  \ar[d]^{\Phi_{LD,LD }}\ar[r]_{\!\!\!\!R\xi_{LD}\ot R\xi_{LD}} & RTLD\ot RTLD \ar[d]^{\Phi_{TLD,TLD }} \ar[rr]_{R\pi_\msd\ot R\pi_\msd}& &RL_\msd\ot RL_\msd\ar[d]^{\Phi_{L_\msd,L_\msd} }\\
       & {R(LD\ot LD) } \ar[r]_{\!\!\!\!\!R(\xi_{LD}\ot \xi_{LD})} & R(TLD\ot TLD) \ar[dd]^{Rm_{TLD} }\ar[rr]_{R(\pi_\msd\ot \pi_\msd)} &&R(L_\msd\ot L_\msd)\ar[dd]^{Rm_{L_\msd} }\\
            & {RL(D\ot D) } \ar[d]^{ RLm_\msd}\ar[u]_{R\Psi_{D,D}} &                      &                        &        \\
 D\ar[r]^{\kappa_D}  \ar@/_2pc/@{->}[rr]^{\eta_D}                                 & {RLD}\ar[r]^{R\xi_{LD}}                          & RTLD\ar[rr]^{R\pi_\msd}         &                    &    RL_\msd 
}
\end{aligned}
\end{equation*}

It remains to prove that $\gamma_\msd$ is $\bar{R}$-universal for $\msd$. Here  we again use elements of the proof given in \cite{AEM}. 

First define, for a $\BC$-monoid $\msa$, a morphism $L_{\bar{R}\msa}\xra{\sigma_\msa }\msa$ in $\Mon\BC$. This will in fact be the  counit of the desired adjunction.
By the  easily checked  identities 
\begin{equation*}
\epsilon_A\circ \alpha_1^{\bar{R}\msa} = \epsilon_A\circ \beta_1^{\bar{R}\msa}\text{\ and \ }
\epsilon_A\circ \alpha_2^{\bar{R}\msa} = \epsilon_A\circ \beta_2^{\bar{R}\msa}.
\end{equation*}
one obtains, using the equivalence \eqref{eq:5} and the universal property of $\pi_{\bar{R}\msa}$, a unique monoid morphism 
$L_{\bar{R}\msa}\xra{ \sigma_\msa }\msa$  making the following diagram commute
\begin{equation}\label{diag:sigma}
\begin{aligned}
\xymatrix@=2.5em{
L_{\bar{R}\msa}     \ar[r]^{\sigma_\msa }&{\msa }   \\
 TLRA\ar[ur]_{\epsilon_A}\ar[u]^{\pi_{\bar{R}\msa}  } 
}
\end{aligned}
\end{equation}

Next one checks, for any    morphism $\msd\xra{ h}\msc$ in $\Mon\BD$, the identities
\begin{eqnarray*}
TLh \circ \alpha_1^\msd = \alpha_1^\msc  & & TLh \circ \beta_1^\msd = \beta_1^\msc\\
TLh \circ \alpha_2^\msd = \alpha_2^\msc\circ L(h\ot h)  &  & TLh \circ \beta_2^\msd = \beta_2^\msc\circ L(h\ot h)
\end{eqnarray*}
These  imply
\begin{eqnarray*}
(\pi_\msc\circ TLh) \circ \alpha_1^\msd& = &\pi_\msc\circ\alpha_1^\msc = \pi_\msc\circ\beta_1^\msc \\ & =& (\pi_\msc\circ TLh) \circ \beta_1^\msd \\
(\pi_\msc\circ TLh) \circ \alpha_2^\msd &=&\pi_\msc\circ\alpha_2^\msc \circ L(h\ot h) =\pi_\msc\circ\beta_2^\msc\circ L(h\ot h)  \\ &=&(\pi_\msc\circ TLh) \circ \beta_2^\msd  
\end{eqnarray*}
Hence, $\pi_\msc\circ TLh $ coequalizes simultaneously the $\BC$-morphisms $\alpha_1, \alpha_2,\beta_1, \beta_2$ and, thus, by Remark \ref{rem:bar} the monoid morphisms $\bar{\alpha}_1, \bar{\alpha}_2,\bar{\beta}_1, \bar{\beta}_2$. Consequently there exist a unique monoid morphism $L_\msd\xra{L_h }L_\msc$ making the following diagram commute
\begin{equation}\label{diag:Eh}
\begin{aligned}
\xymatrix@=2.5em{
  L_\msd  \ar[r]^{ L_h }&{ L_\msc}   \\
  TLD \ar[u]^{\pi_D  }  \ar[r]_{ TLh}          &  TLC  \ar[u]_{\pi_C }
}
\end{aligned}
\end{equation}

Then, for every  morphism $\msd\xra{ d}\bar{R}\msa$ in $\Mon\BD$, the following diagram obviously commutes (in $\BD$) and illustrates the required  one-to-one correspondence between morphisms  $\msd\xra{ d}\bar{R}\msa$ in $\Mon\BD$ and morphisms $L_\msd\xra{\sigma_\msa\circ L_d}\msa$ in $\Mon\BC$.
\begin{equation*}
\begin{aligned}
\xymatrix@=3em{
   D \ar@/^2.2pc/@{->}[rrr]^{\gamma_\msd} \ar@/^1pc/@{->}[rr]^{\eta_D  }\ar@/_3pc/@{->}[rrdd]_d\ar[r]_{\kappa_D}& RLD\ar[r]_{R\xi_{LD}}\ar@/_1pc/@{->}[ddr]_{RLd} & { RTLD} \ar[d]_{ RTLd}\ar[r]_{R\pi_\msd} & RL_\msd \ar[d]^{RL_d}\\
 &  &    RTLRA \ar[r]^{R\pi_{\bar{R}\msa} } \ar[d]_{R\epsilon_A}   &    RL_{\bar{R}\msa}\ar[dl]^{R\sigma_\msa}\\ 
 & & RA\\
}
\end{aligned}
\end{equation*}
 \end{proof}

\begin{remarks}\label{rems:}\rm
To apply this theorem one typically will use Theorem \ref{thm:1}. In this case, if assuming in addition existence of (countable) coproducts in $\BC$ preserved by  $\ot$, Theorem \ref{thm:T} specializes to  Tambara's claim. 
\end{remarks}

\end{document}